\documentclass[oneside]{amsart}

\pdfoutput=1
\usepackage{color}
\usepackage{changepage}
\usepackage[hyphens]{url}
\usepackage{hyperref}
\usepackage[hyphenbreaks]{breakurl}

\usepackage[
    backend=biber,
    style=alphabetic,
    sorting=nyt
]{biblatex}
\usepackage{amsmath}
\usepackage{amssymb}
\usepackage{booktabs}
\usepackage{tikz-cd}
\usetikzlibrary{matrix,arrows,decorations.pathmorphing}

\usepackage{enumitem}
\newlist{step}{enumerate}{1}
\setlist[step]{label=\textbf{Step \arabic*.}}

\usepackage{makecell}

\usepackage{pdflscape}

\newcommand{\ZZ}{\mathbb{Z}} 
\newcommand{\NN}{\mathbb{N}} 
\newcommand{\QQ}{\mathbb{Q}} 
\newcommand{\PP}{\mathbb{P}} 
\newcommand{\KK}{\mathbb{K}} 
\newcommand{\CO}{\mathcal{O}}
\newcommand{\GG}{\mathbb{G}}
\newcommand{\A}{\mathbb{A}}
\newcommand{\val}{\mathfrak{v}}
\newcommand{\lm}{\lambda}
\newcommand{\cX}{\mathcal{X}}

\newcommand{\cF}{\mathcal{F}}
\newcommand{\kf}{Khovanskii-finite}
\newcommand{\pgl}[1]{\mathrm{PGL}(#1,\KK)}

\DeclareMathOperator{\Spec}{Spec}
\DeclareMathOperator{\Proj}{Proj}
\DeclareMathOperator{\ord}{ord}

\DeclareMathOperator{\sing}{Sing}

\newtheorem{theorem}{Theorem}[section]
\newtheorem{proposition}[theorem]{Proposition}
\newtheorem{lemma}[theorem]{Lemma}

\theoremstyle{definition}
\newtheorem{definition}[theorem]{Definition}
\newtheorem{rem}[theorem]{Remark}
\newtheorem{question}[theorem]{Question}
\newtheorem{example}[theorem]{Example}
\newtheorem{construction}[theorem]{Construction}
\newtheorem{algorithm}[theorem]{Algorithm}
\newtheorem{notation}[theorem]{Notation}

\title{Khovanskii-finite rational curves of arithmetic genus 2}
\author{Nathan Ilten}
\address{Department of Mathematics, Simon Fraser University,
8888 University Drive, Burnaby BC V5A1S6, Canada}
\email{\href{mailto:nilten@sfu.ca}{nilten@sfu.ca}}

\author{Ahmad Mokhtar}
\address{Department of Mathematics, Simon Fraser University,
8888 University Drive, Burnaby BC V5A1S6, Canada}
\email{\href{mailto:ahmad\_mokhtar@sfu.ca}{ahmad\_mokhtar@sfu.ca}}
\thanks{Both authors were supported by NSERC}

\date{}

\begin{document}

    \begin{abstract}
        We study the existence of Khovanskii-finite valuations for rational curves of arithmetic genus two. We provide a semi-explicit description of the locus of degree $n+2$ rational curves in $\PP^n$ of arithmetic genus two that admit a Khovanskii-finite valuation. Furthermore, we describe an effective method for determining if a rational curve of arithmetic genus two defined over a number field admits a Khovanskii-finite valuation. This provides a criterion for deciding if such curves admit a toric degeneration.
        Finally, we show that rational curves with a single unibranch singularity are always Khovanskii-finite if their arithmetic genus is sufficiently small.
    \end{abstract}

    \maketitle

    \section{Introduction}

    \subsection{Motivation}
    Toric degenerations appear in a wide range of contexts including mirror symmetry \cite{MS}, numerical algebraic geometry \cite{sottile}, and the study of Seshadri constants \cite{seshadri}. It is thus an important problem to understand which projective varieties have flat degenerations to toric varieties.
    Such degenerations do not necessarily  exist. More specifically, if one fixes a projective embedding of a variety $X\hookrightarrow \PP^n$, and only considers $\GG_m$-equivariant embedded degenerations with respect to any Veronese re-embedding of $X$, a toric degeneration may not exist. This is in particular the case for very general embeddings of smooth curves of genus at least two \cite[pp 2354]{akl}, and very general rational plane quartic curves \cite[Theorem 4.1]{ilten-wrobel}.

    A fruitful approach to constructing toric degenerations is due to D.~Anderson \cite{anderson}.
    Let $X$ be a projective variety over an algebraically closed field $\KK$  with homogeneous coordinate ring $S$.
    We say that a valuation on $S$ is \emph{Khovanskii-finite} if it is homogeneous, and the value semigroup is finitely generated of rank equal to $\dim X+1$, see \S\ref{sec:val} and Definition \ref{def:kfinite}. To any such Khovanskii-finite valuation $\val$, Anderson constructs a degeneration of $X$ to the toric variety whose homogeneous coordinate ring is the semigroup algebra associated to the value semigroup of $\val$ \cite[Theorem 1.1]{anderson}.

    It turns out that Anderson's construction is quite general: every projective $\GG_m$-equivariant flat family $\pi:\cX\to\A^1$ with $\pi^{-1}(1)\cong X$ and $\pi^{-1}(0)$ toric arises from the above-mentioned construction for some embedding of $X$, see e.g.~\cite[Theorem 1.11]{c1degen}. Here, the $\GG_m$-action on $\A^1$ is the standard one.
    The generality of this construction leads to the following question:
    \begin{question}
        \label{q:1}
        Given a projective variety $X=\Proj S$, does its homogeneous coordinate ring $S$ admit a Khovanskii-finite valuation (cf. \cite[Problem 2]{kaveh-manon})?
    \end{question}

    The first author and M.~Wrobel showed that in characteristic zero, although the answer to this question is ``yes'' when $X$ is a rational curve of arithmetic genus $0$ or $1$, there are rational curves of higher arithmetic genus whose homogeneous coordinate rings do not admit a Khovanskii-finite valuation \cite{ilten-wrobel}. In fact, in characteristic zero the homogeneous coordinate ring of a very general rational plane curve of degree at least four does not admit any Khovanskii-finite valuations \cite[Theorem 1.4]{ilten-wrobel}.

    While considering Khovanskii-finiteness for rational curves, the first interesting case is that of arithmetic genus two. This case is the subject of this paper. Before describing our approach and results, we remark that we were surprised to be able to obtain such a complete answer to Question \ref{q:1} in this case. We do not know if this question is even decidable in general!

    \begin{rem}
        It might seem quite special to focus on Question \ref{q:1} for (singular) rational curves. However, this case is in fact one of the most fundamental ones to answering the question in general. By \cite[Corollary 3]{c1degen}, any projective variety $X$ can
        be degenerated to a variety $X'$ equipped with a complexity-one torus action. If $X'$ is rational (for example we are considering a $\QQ$-Gorenstein degeneration of a Fano variety), \cite[Theorem 4.1]{ilten-wrobel} says that finding a toric degeneration of $X'$ amounts to determining if a finite number of (singular) rational curves admit a common Khovanskii-finite valuation. If a toric degeneration of $X'$ exists that is compatible with the torus action on $X'$, we believe this can be combined with the degeneration of $X$ to $X'$ to give a degeneration in one step as in \cite[\S5]{c1degen}.
    \end{rem}

    \subsection{Approach and Results}
    In the remainder of the paper, we will assume that the ground field $\KK$ has characteristic zero.
    Any non-degenerate rational curve $X\subset \PP^n$ of degree $d$ can be realized as a projection from the rational normal curve $C_d\subset \PP^d$ of degree $d$. In fact, if $X$ has arithmetic genus $g$ and $g\leq d-n$, then $X$ is an isomorphic projection of a curve $X'\subset \PP^{d-g}$, see Proposition \ref{prop:project}. Furthermore, a valuation $\val$ on the homogeneous coordinate ring of $X$ is Khovanskii-finite if and only if it is Khovanskii-finite when viewed as a valuation on the homogeneous coordinate ring of $X'$, see Proposition \ref{prop:same}. Since we are focusing on curves with $g=2$ we may thus assume without loss of generality that $n=d-2$, and $X$ is obtained from $C_d$ by projection from a line, see Remark \ref{rem:proj}.

    The possible singularities that can occur in such a curve were classified by the first author together with J.~Buczy\'nski and E.~Ventura \cite[Theorem 1.2]{buczynski-singular}. The singularities are determined by the intersection behavior of the center of projection with linear spaces of $\PP^d$ spanned by collections of osculating flags.

    To address Question \ref{q:1} for such curves, we consider the stratification of the space of such curves by singularity type, or equivalently, by intersection behavior of the center of projection with certain linear spaces in $\PP^d$.
    The techniques of \cite{buczynski-singular} lead to explicit parameterizations for the curves in each stratum. Using the criteria of Khovanskii-finiteness from \cite[Theorem 3.5]{ilten-wrobel}, we are able to describe the locus of curves in each stratum whose coordinate rings admit a Khovanskii-finite valuation as a countable union of Zariski-closed sets, and we write down explicit equations for each such set. See Theorem \ref{mainthm:1} and Remark \ref{rem:locus}.

    While Theorem \ref{mainthm:1} gives a nice description for the locus of Khovanskii-finite rational curves of arithmetic genus $2$, it does not actually answer Question \ref{q:1} in this case: checking infinitely many polynomial identities is not \emph{a priori} possible in finite time. However, we show in Theorem \ref{thm:general-bound} that it is in fact possible to algorithmically determine the answer to this question in our case of study.

    For most strata appearing in Theorem \ref{mainthm:1}, the general curve is not Khovanskii-finite. However, for the strata corresponding to curves with a unique unibranch singularity and no other singularities, the corresponding coordinate rings always admit a Khovanskii-finite valuation. This is in fact a more general phenomenon. We show in Theorem \ref{mainthm:3} that for \emph{any} degree $d$ rational curve $X\subset \PP^n$ with arithmetic genus at most $d/2$, if $X$ has a unique unibranch singularity and no other singularities, the corresponding coordinate ring admits a Khovanskii-finite valuation.

    Throughout, a \emph{curve} $X$ is an integral scheme of dimension one in $\PP^n$; $X$ being \emph{non-degenerate} means it is not contained in any proper linear subspace.

    We now describe the organization of the rest of this paper. In \S \ref{sec:curves}, we introduce notation for rational curves and state some basic results. We continue with a discussion of valuations and Khovanskii-finiteness in \S \ref{sec:val}. In \S \ref{sec:param} we use the techniques of \cite{buczynski-singular} to write down explicit parameterizations for degree $n+2$ curves in $\PP^n$, which we use in \S \ref{sec:kfinite} to determine the locus of Khovanskii-finite curves, proving Theorem \ref{mainthm:1}. In \S \ref{sec:bound}, we show how to decide Question \ref{q:1} for such curves and prove Theorem \ref{thm:general-bound}. The statement and proof of Theorem \ref{mainthm:3} are contained in \S \ref{sec:unibranched}. 

    \subsection*{Acknowledgements}
    This paper was written on the unceded traditional territories of the Squamish, Tsleil-Waututh, Musqueam and Kwikwetlem Peoples. Both authors were supported by NSERC. Portions of this paper are contained in the MSc thesis of the second author. We thank Jake~Levinson for several stimulating questions.
We thank the anonymous referees for their useful comments.

    \section{Preliminaries}\label{sec:prelim}

    \subsection{Rational Curves}\label{sec:curves}
    This section introduces basic notation and results on rational curves. Recall from the introduction that we will always work over an algebraically closed field $\KK$ of characteristic zero. A curve is an integral scheme of dimension one in $\PP^n$. Being non-degenerate means a curve is not contained in any proper linear subspace.

    Let $X\subset\PP^n$ be a rational curve of degree $d$.
    If $X$ is non-degenerate, then necessarily $d\geq n$ and equality holds if and only if $X$ is a rational normal curve of degree $n$.
    A \emph{rational normal curve} of degree $d$ is a curve projectively equivalent to the image of the Veronese embedding $\nu_d:\PP^1\to\PP^d$ given by
    \begin{equation*}
        (x:y)\mapsto(x^d:x^{d-1}y:\ldots:y^d).
    \end{equation*}
    We denote the image of the Veronese embedding by $C_d$.
    Anytime we mention the rational normal curve, we mean $C_d$ unless stated otherwise.

    For us, the importance of the rational normal curves is due to the fact that their projections give all rational curves.
    We briefly discuss these projections.
    Let $V$ be a $d+1$ dimensional vector space over $\KK$.
    Assume $U\subset V$ is a $d-n$ dimensional subspace for some $n<d$.
    The natural linear map $V\to V/U$ induces a rational map
    \begin{equation*}
        \pi:\PP(V)\dashrightarrow\PP(V/U).
    \end{equation*}
    After choosing a basis of $V$ and $V/U$, we may identify $\PP(V)$ with $\PP^d$ and $\PP(V/U)$ with  $\PP^n$.
    In this case we say that $\pi$ is the projection of $\PP^d$ onto $\PP^n$ with center $\PP(U)\subset\PP(V)$.

    We say a projection $\pi:\PP^d\dashrightarrow\PP^n$ is \emph{basepoint free} with respect to a curve $C\subset \PP^d$ if its center does not intersect $C$. If we omit the curve $C$, then we mean that $\pi$ is basepoint free with respect to the rational normal curve $C_d$.
    The projection $\pi$ restricts to a morphism on $C$ if $C$ is smooth, or if $\pi$ is basepoint free with respect to $C$.

    If $\pi$ is basepoint free and $\frac{d}{2}< n < d$, then the induced morphism $\pi:C_d\to\PP^n$ is birational onto its image~\cite[Lemma 6.1, 6.2]{buczynski-singular}.
    A straightforward argument using Bertini's theorem shows that if the morphism $\pi:C_d\to\PP^n$ is birational onto its image, then its image also has degree $d$.

    Conversely, if $X\subset \PP^n$ is a non-degenerate rational curve of degree $d > n$, there exists a basepoint free projection $\PP^d\dashrightarrow\PP^n$ inducing a morphism $\pi:C_d\to \PP^n$ with $\pi(C_d)=X$.
    Indeed, let $\phi:\PP^1\to X$ be the normalization of $X$. The embedding of $\PP^1$ in $\PP^d$ by the full linear system of $\phi^*(\CO_{\PP^n}(1))$ is the rational normal curve $C_d$, and the pullback of the global sections of $\CO_{\PP^n}(1)$ determine the map $\pi:C_d\to\PP^n$.

    \begin{example}[A rational quintic in $\PP^3$]
        \label{ex:rational-quintic}
        Let $V=\KK^6$ and consider the projection $\pi:\PP^5\dashrightarrow\PP^3$ with center $\PP(U)$, where $U=\mathrm{span}\{(1,0,0,0,0,1),(5,1,0,0,0,0)\}$.
        Under an appropriate choice of basis, the map $\pi$ is given by
        \begin{equation*}
            \pi(x_0:\ldots:x_5)=(x_0-5x_1-x_5:x_2:x_3:x_4).
        \end{equation*}
        Since $\PP(U)\subset\PP^5$ does not intersect $C_5$, the projection is basepoint free and the above discussion tells us that $\pi(C_5)$ is a non-degenerate rational curve of degree 5 in $\PP^3$.
        We will come back to this example in later sections.
    \end{example}

    We now recall compact notation for the homogeneous coordinate ring of a rational curve.
    \begin{notation}[{See~\cite[{Construction 3.1}]{ilten-wrobel}}]
        Let $L\subset\KK[x,y]_d$ be an $m$-dimensional linear subspace where $\KK[x,y]_d$ is the $\KK$-vector space of homogeneous polynomials of degree $d$ in $x$ and $y$.
        Define
        \begin{equation}
            \label{eq:RL}
            R(L)=\bigoplus_{k\geq0}L^k,
        \end{equation}
        where $L^k$ denotes all linear combinations of elements of the form $\ell_1\cdots \ell_k$ with $\ell_i\in L$.
        Since $L^{k_1}L^{k_2}\subset L^{k_1+k_2}$, $R(L)$ is naturally a finitely generated graded $\KK$-algebra.
        A basis for $L$ provides a set of algebra generators for $R(L)$.
    \end{notation}
    This connects to the discussion on projections as follows. The rational normal curve $C_d$ is embedded in $\PP(V)$, where $V$ is the vector space dual to $\KK[x,y]_d$. The choice of some $U\subset V$ determines $L\subset V^*$ as the orthogonal complement of $U$:
    \[
        L=U^\perp\subset \KK[x,y]_d.
    \]
    It follows from the description of $R(L)$ that it is the homogeneous coordinate ring for $\pi(C_d)$, where $\pi$ is the projection with center $\PP(U)$. Conversely, given some $0\neq L\subset \KK[x,y]_d$, we obtain a projection $\pi$ from $\PP(U)$, where $U=L^\perp\subset \KK[x,y]_d^*=V$. In the remainder of the paper, we freely switch between both points of view.

    We state a proposition that will let us reduce from the case of arithmetic genus two rational curves to curves in $\PP^n$ of degree $n+2$. In our algorithm for determining Khovanskii-finiteness, we will be concerned about fields of definition, so we first introduce some terminology.
 Consider a basepoint free projection $\pi:C\to X$ of rational curves and a number field $F$. We say that $\pi$ has \emph{ramification defined over $F$} if $C$, $X$, and $\pi$ are defined over $F$, as is every point in $\pi^{-1}(\sing (X))$. Here, $\sing(X)$ is the singular locus of $X$.
    \begin{proposition}
        \label{prop:project}
        Let $C$ be a rational curve and $\pi:\PP^d \dashrightarrow \PP^n$ be a projection which is basepoint free with respect to $C$. Let $X=\pi(C)$, and denote the arithmetic genera of $C$ and $X$ by $g(C)$ and $g(X)$. Assume that $\pi:C\to X$ is birational and $g(X)-g(C)\leq d-n$. Then $\pi$ factors as a composition $\pi'\circ\pi''$ of basepoint free projections
        \[
            \begin{tikzcd}
                \PP^d \arrow[r,dashrightarrow]& \PP^{n'} \arrow[r,dashrightarrow]&\PP^n\\
                C \arrow[r,"\pi''"] \arrow[u,hookrightarrow]& X'\arrow[r,"\pi'"] \arrow[u,hookrightarrow]&X\arrow[u,hookrightarrow]
            \end{tikzcd}
        \]
        where $n'=d+g(C)-g(X)$ and $\pi':X'\to X$ is an isomorphism.
        In particular, any degree $d$ curve $X\subset \PP^n$ with arithmetic genus $g$ satisfying $g\leq d-n$ is obtained via an isomorphic projection from a degree $d$ curve $X'\subset \PP^{n'}$, with $n'=d-g$.

Furthermore, if $\pi$ has ramification defined over a number field $F$, then $\pi''$ may be chosen so that it does as well.

\end{proposition}
    \begin{proof}
        We prove the existence of a decomposition $\pi=\pi'\circ \pi''$ by induction on $g(X)-g(C)$. If this difference is zero, then since $\pi:C\to X$ is birational, it must be an isomorphism, and the claim follows. Suppose that $g(X)-g(C)>0$, so $\pi$ is not an isomorphism. Then $\pi$ fails to separate either points or tangent vectors of $C$, so the center of projection $\PP(U)$ intersects either a secant line of $C$, or a tangent line of $C$ at some point. Let $u\in\PP^d$ be such a point of intersection. Then
        letting $\phi''$ be the projection with center $u$, we may factor $\pi$ as the composition $\phi'\circ \phi''$ of $\phi''$ with another projection.

	If $\pi$ has ramification defined over $F$, then the secant or tangent line of $C$ intersected by $\PP(U)$ is defined over $F$, hence so is $u$. It follows that $\phi'$ and $\phi''$ have ramification defined over $F$. 

        By construction, $\phi''$ does not restrict to an isomorphism on $C$, so the arithmetic genus $g(C')$ of $C'=\phi''(C)$ is larger than $g(C)$. The projection $\phi':\PP^{d-1}\dashrightarrow\PP^n$ is basepoint free with respect to $C'$, and restricts to a birational morphism on $C'$. Furthermore, $g(X)-g(C')<g(X)-g(C)$ and $g(X)-g(C')\leq (d-1)-n$, so by induction, $\phi'$ factors as in the statement of the proposition. Composing the first morphism in this factorization with $\phi''$ completes the induction step, showing the existence of a decomposition 
$\pi=\pi'\circ \pi''$. The statement regarding $\pi''$ having ramification over $F$ also follows by induction, since the intermediate morphisms $\phi'$ and $\phi''$ had ramification defined over $F$.

    We now consider the claim that any degree $d$ curve $X\subset \PP^n$ with arithmetic genus $g$ satisfying $g\leq d-n$ is obtained via an isomorphic projection from a degree $d$ curve $X'\subset \PP^{n'}$, with $n'=d-g$.
We use the observation from above that $X$ is the basepoint free birational projection $\pi$ of the rational normal curve $C_d$. The claim now follows from the above factorization of $\pi$.
    \end{proof}

    \subsection{Valuations and Khovanskii-Finiteness}\label{sec:val}
    This section reviews basic definitions and results about valuations.
    For details on valuation theory, see e.g.~\cite[Ch.\ VI]{zariski-samuel-ii}.

    \begin{definition}
        Let $S$ be a finitely generated $\KK$-domain and let $(\Gamma,>)$ be a totally ordered finitely generated abelian group.
        A \emph{$\KK$-valuation} on $S$ (with values in $\Gamma$) is a map $\val:S\backslash\{0\}\to\Gamma$ satisfying
        \begin{enumerate}
            \item $\val(c)=0$ for $c\in\KK\backslash\{0\}$,
            \item $\val(fg)=\val(f)+\val(g)$ for $f,g\in S\backslash\{0\}$,
            \item $\val(f+g)\geq\min \{\val(f),\val(g)\}$ for $f,g\in S\backslash\{0\}, f+g\neq 0$.
        \end{enumerate}
        The image of $\val$ is a subsemigroup of $\Gamma$ called the \emph{value semigroup}.
        The \emph{rank} of $\val$ is the rank of the subgroup of $\Gamma$ generated by the value semigroup.
    \end{definition}
    \noindent The rank of a valuation on $S$ is at most the Krull dimension of $S$~\cite[Ch. VI, \textsection10]{zariski-samuel-ii}.
    A valuation $\val$ has \emph{full rank} if its rank equals the Krull dimension.
    Any valuation $\val$ on a $\KK$-domain $S$  induces a valuation on its field of fractions $F$, and hence on any other $\KK$-domain contained in $F$. We will denote all these valuations with $\val$.

    Since our focus is on homogeneous coordinate rings of projective varieties, we will in particular be interested in homogeneous valuations.
    Consider a finitely generated graded $\KK$-domain $S$. A valuation $\val$ on $S$ is \emph{homogeneous}
    if for every nonzero $f\in S$ we have
    \begin{equation*}
        \val(f)=\min_{f_k\neq0}\{\val(f_k)\},
    \end{equation*}
    where $f=\sum_{k}f_k$ is the decomposition of $f$ into its homogeneous parts.

    \begin{definition}[Khovanskii-finiteness]
        \label{def:kfinite}
        A valuation $\val$ on a finitely generated homogeneous $\KK$-domain $S$ is \emph{Khovanskii-finite} with respect to $S$ if $\val$ is homogeneous, has full rank, and the value semigroup is finitely generated.
    \end{definition}
    \noindent If $X$ is an embedded projective variety with homogeneous coordinate ring $S$, we say that $\val$ is \kf{} with respect to $X$ if  $\val$ is \kf{} with respect to $S$.

    Given a finitely generated homogeneous $\KK$-domain $S$, how can we determine if it admits a \kf{} valuation?
    This is particularly of interest when $S$ is the homogeneous coordinate ring of a projective variety $X$, because then one can determine whether $X$ admits a toric degeneration (see Question~\ref{q:1} and the preceding discussion).
    Homogeneous valuations for coordinate rings of rational curves are particularly easy to understand:
    \begin{construction}[{\cite[\textsection3]{ilten-wrobel}}]\label{const:val-family}
        Let $L\subset \KK[x,y]_d$ be a linear subspace and fix a point $Q=(\alpha:\beta)\in\PP^1$.
        We obtain a homogeneous valuation $\val_Q$ on \[R(L)=\bigoplus_{k\geq0}L^k\] via
        \begin{equation*}
            \val_Q:f\mapsto \left(k, \ord_Q(f)\right)\quad\text{for}\quad f\in L^k\backslash\{0\},
        \end{equation*}
        where $\ord_Q(f)$ denotes the highest power of $\beta x-\alpha y$ that divides $f$.
        Here $\ZZ^2$ is ordered lexicographically.
    \end{construction}
    \noindent Every full rank homogeneous valuation on $R(L)$ will have a semigroup isomorphic to the value semigroup of $\val_Q$ for some $Q\in \PP^1$ \cite[\S 5.2]{ilten-manon}. Thus, as far as Khovanskii-finiteness for rational curves is concerned, we only need to consider valuations of the form $\val_Q$.

    \begin{rem}\label{rem:flag}
	    Let $Y$ be a $d$-dimensional variety.
	    In Newton-Okounkov theory, it is customary to construct valuations on $\KK(Y)$ coming from a full flag $\cF_0\subset \cF_1\subset \ldots \subset \cF_{d-1}\subset Y$ on a variety $Y$, see \cite[Example 2.13]{no1} or \cite[\S 1.1]{no2} for details. The valuation $\val_Q$ arises in this fashion as follows.

	  In our setting, we let $Y\to \A^2$ be the blowup of $\Spec \KK[x,y]=\A^2$ in the origin, with exceptional divisor $E$. Let $\widehat Q$ be the line in $\A^2$ corresponding to $Q\in \PP^1$, with strict transform $\widehat{Q}'$. Then $\val_Q$ is the valuation corresponding to the flag $\cF_0=\widehat{Q}'\cap E\subset \cF_1=E\subset Y$.
    \end{rem}

    There is a straightforward (albeit non-effective) necessary and sufficient condition for $\val_Q$ to be \kf{}:
    \begin{theorem}[{\cite[Thm.\ 3.5]{ilten-wrobel}}]
        \label{thm:kf-criterion}
        Let $L\subset \KK[x,y]_d$ be a linear subspace.
        Then for a point $Q=(\alpha:\beta)\in\PP^1$, $\val_Q$ is \kf{} with respect to $R(L)$ if and only if there exists $k\in \NN$ such that $(\beta x - \alpha y)^{dk}\in L^k$.
    \end{theorem}

    \begin{example}
        \label{ex:val-rational-quintic}
        Let $L\subset \KK[x,y]_5$ be the subspace spanned by \[\{x^5-5x^4 y-y^5,x^3 y^2,x^2 y^3,xy^4\}.\]
        Then $S=R(L)=\KK[x^5-5x^4 y-y^5,x^3 y^2,x^2 y^3,xy^4]$ is the homogeneous coordinate ring of the rational quintic curve from Example \ref{ex:rational-quintic}.
        Using induction, one can show that for $k\in\NN$, a basis for $L^k$ is given by
        \begin{equation*}
            \{x^{5k}-5kx^{5k-1}y+(-1)^k y^{5k},x^{5k-2}y^2,\ldots,xy^{5k-1}\}.
        \end{equation*}
	See Example \ref{ex:cwsb}.

        Since $x^{5k}\notin L^k$ for $k\geq1$, we see by the previous theorem that $\val_{(0:1)}$ is not \kf{}. In fact, the value semigroup is generated by
        \begin{equation*}
            \{(5,0),(5,1),(5k,5k-3),(5k,5k-2)\,|\,k\in\NN\}.
        \end{equation*}

        On the other hand, $(x-y)^5\in L$ which means that the valuation $\val_{(1:1)}$ is \kf{}.
        By using the above basis one can show that the value semigroup of $\val_{(1:1)}$ is generated by $\{(5,0),(5,1),(5,2),(5,5)\}$.
    \end{example}
    \begin{rem}
    Theorem \ref{thm:kf-criterion} may be interpreted geometrically as follows. Let $\pi:\PP^d\dashrightarrow \PP^n$ be the projection with center $\PP(L^\perp)$, with $X=\Proj (R(L))\subset \PP^n$ and $C_d\subset \PP^d$ the rational normal curve. As long as $\PP(L^\perp)$ avoids $C_d$, the induced morphism $\pi:C_d\to X$ is just the normalization map. Then $\val_Q$ is \kf{} with respect to $R(L)$ if and only if there exists a hypersurface $H\subset \PP^n$ such that the intersection 
    \[\overline{\pi^{-1}(H)}\cap C_d\] is set-theoretically the point $Q\in \PP^1\cong C_d$. The degree of the hypersurface $H$ is exactly the natural number $k$ in the statement of Theorem \ref{thm:kf-criterion}.

    We may compare this with the criterion of \cite[Example 4.2]{anderson} (see also \cite[Proposition 14]{akl}). Let $X\subset \PP^n$ be a  smooth curve with homogeneous coordinate ring $R$, and fix a point $Q\in X$. This point determines a valuation $\val$ on $R$, see \cite{no2}.
    Then as noted in \cite[Example 4.2]{anderson}, the value semigroup $\val(R)$ is finitely generated \emph{by elements of degree one} if and only if there exists a hyperplane $H\subset\PP^n$ such that the intersection $X\cap H$ is set-theoretically equal to the point $Q$. More generally, the value semigroup $\val(R)$ is finitely generated if and only if there exists a hypersurface $H\subset\PP^n$ such that the intersection $X\cap H$ is set-theoretically equal to the point $Q$.
\end{rem}

    Theorem~\ref{thm:kf-criterion} hints at an algorithmic approach to determining if domains of the form $R(L)$ admit a Khovanskii-finite valuation:
    \begin{algorithm}[Naive algorithm]\label{alg:naive}
        Let $L$ and $Q$ be as in Theorem \ref{thm:kf-criterion}.
        To verify whether $\val_Q$ is \kf{} with respect to $R(L)$, test the containment $(\beta x-\alpha y)^{dk}\in L^k$ for all natural numbers $k$.
    \end{algorithm}
    The reason we call this algorithm naive is that it seems to work only if $\val_Q$ is \kf{} in which case we hit a value $k$ such that the containment holds and the algorithm stops.
    Otherwise, the process never terminates.
    Nonetheless, we will show that for rational curves with arithmetic genus 2, there is an upper bound for the exponents $k$ that we need to test (see Theorem \ref{thm:general-bound}).
    This means that the naive algorithm can be made effective.

    We conclude the section by showing that for rational curves of arithmetic genus two, we can reduce to the case of degree $n+2$ curves in $\PP^n$:
    \begin{proposition}
        \label{prop:same}
        Let $X$ and $X'$ be as in Proposition \ref{prop:project}. Then $\val_Q$ is \kf{} with respect to $X$ if and only if $\val_Q$ is \kf{} with respect to $X'$.
    \end{proposition}
    \begin{proof}
        Let $S$ and $S'$ be the homogeneous coordinate rings of $X$ and $X'$, respectively. We view both as subrings of $\KK[x,y]$, and have that $S\subset S'$. Since $X$ and $X'$ are isomorphic, they have the same arithmetic genus, and thus the same Hilbert polynomial. It follows that $S_k=S_k'$ for $k$ sufficiently large. The claim now follows directly from Theorem \ref{thm:kf-criterion}.
    \end{proof}

    \begin{rem}
        \label{rem:proj}
        Suppose that $X$ is any non-degenerate rational curve of arithmetic genus two. Using Propositions \ref{prop:project} and \ref{prop:same}, we may assume without loss of generality after replacing $X$ by $X'$ that $X\subset \PP^n$ is a curve of degree $d=n+2$. Indeed, to apply Proposition \ref{prop:project}, we need to start with a curve $X$ with $d\geq n+2$. This will always be the case, since if $d=n+1$, the resulting curve is obtained via projection from a point and can have at worst a simple node or cusp.
    \end{rem}

    \section{Rational curves of degree $n+2$ in $\PP^n$}
    In this section we study \kf{} valuations on non-degenerate rational curves of degree $n+2$ and arithmetic genus 2 in $\PP^n$. 
    To do this we consider the stratification of the space of such curves by singularity type.
    In each stratum we specify a set of \emph{representative curves} such that any curve in the stratum is projectively equivalent to a representative curve.
    Because admitting a \kf{} valuation depends only on the homogeneous coordinate ring, it is enough to study \kf{} valuations on the representative curves.
    The techniques of~\cite{buczynski-singular} lead to explicit parameterization for the curves in each stratum.
    We use this parameterization to describe the \kf{} valuations for each representative curve.

    We will always implicitly assume that $n\geq 3$, since plane quartics have arithmetic genus three.

    \subsection{Parameterization}\label{sec:param}
    Let $n\geq 3$ and $d=n+2$. As discussed in \S\ref{sec:curves}, a basepoint free projection $\pi$ of the rational normal curve yields a non-degenerate rational curve $X\subset \PP^n$ of degree $d$. Furthermore, every such curve arises this way. Hence, in order to parameterize all non-degenerate rational curves in $\PP^n$ of degree $d$, we vary the center of projection $\PP(U)$ among all possible $(d-n-1)$-dimensional linear subspaces of $\PP^d$ which do not intersect the rational normal curve $C_d$. Here $d-n-1=1$, so we are considering lines in $\PP^d$ as the center of projection.

Next we turn to the singularities that can occur on $X$ and the conditions that each singularity type imposes on the center of projection.
First we discuss osculating spaces and their associated multifiltrations, see \cite[\S3.1]{buczynski-singular}.
We fix a point $P\in C_d\subset \PP^d$. Let $W^i(P)\subset W=H^0(C_d,\CO_{C_d}(1))$ be the subspace of sections vanishing to order at least $i$ at $P$. Here, $W$ can be identified with the space of degree $d$ binary forms.
    For $0\leq i\leq d$, the $i$-th \emph{osculating space} of the rational normal curve $C_d$ at $P$ is defined to be the subspace $V^i(P)=\left( W^i(P) \right)^{\perp}$ of $V=W^*$.

    The decreasing flag of subspaces
    \[
        W=W^0(P)\supset W^1(P) \supset \cdots \supset W^d(P)=0,
    \]
    gives rise to an increasing flag of osculating spaces
    \[
        0=V^0(P) \subset V^1(P) \subset \cdots \subset V^d(P)=V.
    \]
    Combining osculating spaces at various points on $C_d$ leads to a multifiltration:
    \begin{definition}[{see~\cite[\S3.1]{buczynski-singular}}]
        Let $P_1,\ldots,P_r\in C_d$ be $r$ distinct points.
        These determine a $\ZZ^r$-\emph{graded multifiltration} $\cF^\bullet(P_1,\ldots,P_r)$ of $V=H^0(C_d,\CO_{C_d}(1))^*$ where for $\alpha=(\alpha_1,\ldots,\alpha_r)\in\ZZ^r_{\geq 0}$ we set
        \[
            \cF^{\alpha}(P_1,\ldots,P_r)=\left\langle V^{\alpha_1}(P_1),\ldots, V^{\alpha_r}(P_r) \right\rangle.
        \]
        Here, $\langle \bullet \rangle$ denotes the span as a $\KK$-vector space.
    \end{definition}

    \begin{table}
        \begin{tabular}{ l l l }
            \toprule
            \textbf{Singularity}&\textbf{Delta-invariant}& \textbf{Conditions on $\lm'$}\\
            \midrule
            Cusp &1& $\lm'(2)=\lm'(4)= 1 $  \\
	    \addlinespace
            Node &1& $\lm'(1,1)=\lm'(2,2)=1$ \\
            \midrule
            (3,4,5)-cusp &2& $\lm'(3)=\lm'(5)=2$ \\
	    \addlinespace
            (2,5)-cusp &2& $\lm'(2)=\lm'(3)=1,\ \lm'(4)=2$\\
	    \addlinespace
            Tacnode &2& $\lm'(1,1)=\lm'(1,2)=\lm'(2,1)=1$,\\
            && $\lm'(2,2)=2$\\
	    \addlinespace
            Cusp with smooth branch &2& $\lm'(2,1)=\lm'(4,2)=2$\\
	    \addlinespace
            Ordinary triple point &2& $\lm'(1,1,1)=\lm'(2,2,2)=2$\\
            \bottomrule
        \end{tabular}
	    \vspace{.5cm} 
        \caption[Possible singularities of rational curves of degree $n+2$ in $\PP^n$ for $n\geq 3$]{Singularities of $X$ and their $\lm'$ condition in Theorem~\ref{thm:singularities-lambda}. Notation taken from \protect\cite{buczynski-singular}.}
        \label{tab:singularity-types}
    \end{table}

    A major result of \cite{buczynski-singular} is that in our setting, the singularities occurring from projecting $C_d$ are determined by the intersection behavior of the center of projection with a multifiltration of osculating spaces.

    \begin{theorem}[{\cite[Theorems 1.2, 6.3]{buczynski-singular}}]
        \label{thm:singularities-lambda}
        Let $d=n+2$, and $\PP(U)\subset \PP^d$ be a line with  $\PP(U)\cap C_d=\emptyset$.
        Consider the projection $\pi:C_d\to X$ with center $\PP(U)$ as described above.
        \begin{enumerate}
            \item For a point $P\in C_d$, the image $\pi(P)$ is a singular point of $X$ if and only if $\PP(U)$ intersects a secant or tangent line of $C_d$ passing through $P$.
            \item The singularities that may occur on $X$ are listed in Table~\ref{tab:singularity-types}.
            If $Q\in X$ is a singular point with preimage $\{P_1,\ldots,P_r\}\subset C_d$, the singularity type of $Q$ is determined from the conditions of Table~\ref{tab:singularity-types} applied to the function
            \begin{align*}
                \lm':\ZZ^r_{\geq 0}&\longrightarrow \ZZ_{\geq 0}\\
                \alpha &\longmapsto \dim \left(U \cap \cF^{\alpha}(P_1,\ldots,P_r)\right).
            \end{align*}
    \item For any tuple of singularity types in Table~\ref{tab:singularity-types}, the space of all lines $\PP(U)$ parameterizing projections with exactly these singularities is non-empty if and only if the sum of the delta-invariants is at most 2.\label{item:three}
        \end{enumerate}
    \end{theorem}

    If $X\subset\PP^n$ is a non-degenerate rational curve of degree $n+2$ and arithmetic genus 2, the delta-invariants of the singularities of $X$ sum to 2.
    By the above theorem, the possible configurations of singularities that $X$ can have are either exactly one of the delta-invariant 2 singularities of Table~\ref{tab:singularity-types}, or two nodes, two cusps, or a node and a cusp.

    \begin{example}[Tacnode]\label{ex:tacnode-param}
        Let $X\subset\PP^n$ be a non-degenerate rational curve of degree $d=n+2$ with one singular point $Q\in X$ that is a tacnode.
	Part \ref{item:three} of the above theorem asserts that such a curve exists.
        We use the data in Table~\ref{tab:singularity-types} to characterize the center of projection $\pi:\PP^{d}\dashrightarrow\PP^n$ with $\pi(C_{d})=X$.
        This then gives an explicit parameterization for $X$.
        The preimage $\pi|_{C_{d}}^{-1}(Q)$ consists of two points $P_1,P_2\in C_{d}$.
        After acting by $\pgl{2}$ on $C_{d}$, we can assume $P_1=(1:0:\ldots:0)$ and $P_2=(0:\ldots:0:1)$.
        Table~\ref{tab:singularity-types} lists the following conditions on the function $\lm':\alpha\mapsto\dim\left(U\cap\cF^{\alpha}(P_1,P_2)\right)$,
        \[
            \lm'(1,1)=\lm'(1,2)=\lm'(2,1)=1,\quad\lm'(2,2)=2.
        \]
	We will use the ordered basis $\{e_0,\ldots,e_{n+2}\}=\{x^{n+2},x^{n+1}y,\ldots,y^{n+2}\}$ for \[W=H^0(C_d,\CO_{C_d}(1))=H^0(\PP^1,\CO_{\PP^1}(n+2))\] and the corresponding dual basis $\{e_0^*,\ldots,e_{n+2}^*\}$ for $V=W^*$.
        Then we have
        \begin{align*}
            &W^1(P_1)=\{e_1,\ldots,e_{n+2}\},\quad &W^1(P_2)&=\{e_0,\ldots,e_{n+1}\},\\
            &W^2(P_1)=\{e_2,\ldots,e_{n+2}\},\quad &W^2(P_2)&=\{e_0,\ldots,e_{n}\}.\\
        \end{align*}
        and dually
        \begin{align*}
            &V^1(P_1)=\{e_0^*\},\quad &V^1(P_2)&=\{e_{n+2}^*\},\\
            &V^2(P_1)=\{e_0^*,e_1^*\},\quad &V^2(P_2)&=\{e_{n+1}^*,e_{n+2}^*\}.\\
        \end{align*}
        The condition $\lm'(1,1)=1$ is equivalent to the condition that the line $\PP(U)\subset\PP(V)$, which is the center of the desired projection, intersects the secant line $P_1 P_2$.
        The element of $\pgl{2}$ we used above can be adjusted so that this point of intersection in $\PP(U)$ lifts to the point  $(1,0,\ldots,0,1)\in U$.
        Since $U$ is contained in $\cF^{(2,2)}=\langle V^2(P_1),V^2(P_2) \rangle$,
$U$ must be the span of $(1,0,\ldots,0,1)$ and a point of the form $(z_1,z_2,0,\ldots,0,z_3,z_4)$ for some $z_1,z_2,z_3,z_4 \in\KK$.
        The conditions $\lm'(1,2)=\lm'(2,1)=1$ imply that $z_2$ and $z_3$ are nonzero.
        After row-reducing and relabeling we can express $U$ as
        \[
            U=\mathrm{span}\{(1,0,\ldots,0,1),(0,a,0,\ldots,0,1,b)\}\quad(a\neq 0).
        \]
        By computing a basis for $L=U^\perp \subset W$, we obtain the parameterization
        \begin{align*}
            X=\Big\{(x^{n+2}+bxy^{n+1}-y^{n+2}:x^{n+1}y-axy^{n+1}:\\
            \qquad x^{n}y^2:\ldots:x^2 y^{n})\in \PP^n \,\Big|(x:y)\in\PP^1 \Big\}\quad (a\neq0),
        \end{align*}
        with coordinate ring  $R(L)=\bigoplus_{k\geq0}L^k$.
        This is a family of curves parameterized by an open subset of $\A^2$ $(a,b\in\KK, a\neq 0)$.
         Any non-degenerate rational curve of degree $n+2$ in $\PP^n$ having a tacnode is projectively equivalent to a curve in this family.
    \end{example}

    \begin{rem}
        The family obtained in the previous example contains redundant curves, i.e.\ curves that are projectively equivalent.
	For example let $\zeta$ be a primitive $(n+2)$-th root of unity. For every $a,b\in\KK$, the two curves corresponding to parameter values $(a,b)$ and $(a',b')=(a\zeta^2,b\zeta)$ are projectively equivalent.
        This is due to the fact that the $\pgl{2}$ element we used to move points on $C_d$ is not unique.
        Since our goal is to study Khovanskii-finiteness on these curves, this redundancy is not important.
    \end{rem}

    \begin{example}[Two nodes]\label{ex:two-nodes-param}
        Let $X\subset\PP^n$ be a rational curve of degree $d=n+2$ having two nodal singularities $Q,Q'\in X$.
        Assume $\pi:\PP^{d}\dashrightarrow\PP^n$ is the projection with $\pi(C_d)=X$.
        This singularity configuration emerges if the center of projection intersects two secant lines of the rational normal curve $C_d$.
        Hence there are four points $P_1,P_2,P'_1,P'_2\in C_d$ such that $\pi|_{C_d}^{-1}(Q)=\{P_1,P_2\}$ and $\pi|_{C_d}^{-1}(Q')=\{P'_1,P'_2\}$.
        If $P_i=P'_j$ for some $1\leq i,j\leq 2$ then all four points map to $Q=Q'$ and we end up with an ordinary triple point.
        Thus $P_1,P_2,P'_1,P'_2$ are distinct.

        After acting by $\pgl{2}$ we may assume that $P_1=(1:0:\ldots:0),P_2=(0:\ldots:0:1)$ and $P'_1=(1:1:\ldots:1)$.
        The last point $P'_2$ will be of the form $(1:c:c^2:\ldots:c^{n+2})$ for some $c\in\KK$ with $c\neq 0,1$.
        Using the same notation and ordered bases as in the previous example for $W$ and $V=W^*$, we can express $U\subset V$ as
        \[
            U=\mathrm{span}\{(1,0,\ldots,0,a),(1+b,c+b,c^2+b,\ldots,c^{n+2}+b)\}\quad(c\neq0,1),
        \]
        for some $a,b\in\KK$.
        It is straightforward to show that $a,b\neq0$ if and only if the center of projection $\PP(U)$ does not meet $C_d$.

        Similar to the previous example, by computing a basis for the orthogonal complement of $U$ we get an explicit parameterization for the curve $X$.
        If $b+c\neq0$ we have
        \begin{multline}\label{eq:paramTwoNodesCase1}
        C=\Big\{
        (-ax^{n+2}-\frac{c^{n+2}-ab-a+b}{b+c}x^{n+1}y+y^{n+2}:-\frac{b+c^2}{b+c}x^{n+1}y+x^n y^2: \\
            -\frac{b+c^3}{b+c} x^{n+1}y+x^{n-1}y^3:\ldots: -\frac{b+c^{n+1}}{b+c}x^{n+1}y+xy^{n+1})\in \PP^n \,\Big|(x:y)\in\PP^1
        \Big\}.
        \end{multline}
        If $b+c=0$ we end up with
        \begin{multline}\label{eq:paramTwoNodesCase2}
        C=\Big\{
        (-ax^{n+2}-\frac{c^{n+2}+ac-a-c}{c^2-c}x^n y^2+y^{n+2}:x^{n+1} y: \\
            -\frac{c^3-c}{c^2-c} x^n y^2+x^{n-1}y^3:\ldots:-\frac{c^{n+1}-c}{c^2-c}x^n y^2+xy^{n+1})\in \PP^n \,\Big|(x:y)\in\PP^1\Big\}.
        \end{multline}
        In this way, we obtain a family of curves parameterized by an open subset of $\A^3$ (parameters $a,b,c$ with conditions $a,b\neq0,\,c\neq0,1$) which is a family of representative curves for non-degenerate rational curves of degree $n+2$ in $\PP^n$ having two nodes.
    \end{example}

    The method of the previous two examples can be used to write down explicit parameterizations for all singularity configurations that can arise on non-degenerate rational curves of degree $n+2$ in $\PP^n$ having arithmetic genus 2.
    We summarize the result of such computations in Table~\ref{tab:param}.

	    \begin{table}
    \begin{adjustwidth}{-2cm}{-2cm}
\centering
            \begin{tabular}{ l l l }
                \toprule
		\bf \makecell{ Singularity\\Type}  & \bf Parameters & \makecell{\bf Curve Parameterization} \\
                \midrule
\addlinespace
                Tacnode &  \makecell{$a,b$ \\ $a\neq0$ } &
                \begin{tabular}{l}
                    $(x^{n+2}+bxy^{n+1}-y^{n+2}:x^{n+1}y-axy^{n+1}:$\\$x^{n}y^2:\ldots:x^2 y^{n})$\\
                \end{tabular}\\

\addlinespace
\midrule
\addlinespace
                (3,4,5)-cusp &  $a$ &
                \makecell{
                    $(x^{n+2}+ax^n y^2:x^{n-1}y^3:x^{n-2}y^4:\ldots:y^{n+2})$\\
                }\\

\addlinespace
\midrule
\addlinespace

                (2,5)-cusp &  $a,b$ &
                \makecell{
                    $(x^{n+2}-ax^{n-1}y^3:x^n y^2-bx^{n-1}y^3:x^{n-2}y^4:\ldots:y^{n+2})$\\
                }\\

\addlinespace
\midrule
\addlinespace

                \makecell{Cusp with\\ smooth branch} &  $a$ &
                \makecell{
                    $(x^{n+2}-ax^{n+1} y-y^{n+2}:x^n y^2:x^{n-1} y^3:\ldots:xy^{n+1})$\\
                }\\

\addlinespace
\midrule
\addlinespace

                \makecell{Ordinary\\ triple point} &   \makecell{$a,b$ \\$a\neq0$\\ $a+b\neq-1$} &
                \makecell{
                    $(x^{n+2}+bxy^{n+1}+ay^{n+2}:x^{n+1} y-xy^{n+1}:$\\$x^n y^2-x y^{n+1}:\ldots:x^2 y^n-xy^{n+1})$\\
                }\\

\addlinespace
\midrule
\addlinespace

                Two cusps &
                \makecell{$a,b$} &
                \makecell{
                    $(x^{n+2}-ax^{n+1}y:x^n y^2:\ldots:x^2 y^n:-bxy^{n+1}+y^{n+2})$\\
                }\\

\addlinespace
\midrule
\addlinespace

                \makecell{One cusp,\\ one node} &  \makecell{$a,b$ \\ $b+1\neq0$} &
                \makecell{
                    $(x^{n+2}-ax^{n+1}y+(a-1)xy^{n+1}:x^n y^2-xy^{n+1}:\ldots:$\\$x^2 y^n-xy^{n+1}:bxy^{n+1}+y^{n+2})$\\
                }\\

\addlinespace
\midrule
\addlinespace
                Two nodes &  \makecell{$a,b,c$ \\ $a\neq0$\\  $b\neq 0$ \\ $c\neq0,1$} &
                \makecell{Equation\ \eqref{eq:paramTwoNodesCase1} if $b+c\neq0$, \\ Equation\ \eqref{eq:paramTwoNodesCase2} if $b+c=0$.}\\
\addlinespace
\bottomrule
	\end{tabular}
	    \vspace{.5cm} 
	    \caption{Parameterization of the representative curves of non-degenerate rational curves of degree $n+2$ in $\PP^n$ having arithmetic genus 2 ($n\geq3$).}   \label{tab:param}
\end{adjustwidth}
        \end{table}

    \subsection{Khovanskii-finiteness}\label{sec:kfinite}
    In \S\ref{sec:param} we grouped together non-degenerate rational curves of degree $d=n+2$ in $\PP^n$ having arithmetic genus 2 based on their singularity configurations.
    In each family, we parameterized a set of representative curves in Table \ref{tab:param}.
    Now in each representative family, we derive explicit conditions for valuations to be \kf{}.

    \begin{theorem}
        \label{mainthm:1}
        Let $n\geq3$ and consider the stratification of the space of all non-degenerate rational curves of degree $d=n+2$ in $\PP^n$ having arithmetic genus 2 based on singularity type.
        \begin{enumerate}
		\item For each of the eight strata, explicit parameterizations for a family of representative curves are listed in Table~\ref{tab:param}.\label{part:uno}
		\item For each curve in the eight representative families in \ref{part:uno}, the valuations $\val_Q$ ($Q\in\PP^1$) that are \kf{} are listed in Table~\ref{tab:loci}.\label{part:dos}
        \end{enumerate}
    \end{theorem}
    \begin{table}
        \centering
        \begin{tabular}{ l l }
            \toprule
            \bf \makecell{Singularity\\Type} & \bf \makecell{$(\alpha:\beta)\in\PP^1$ with $\val_{(\alpha:\beta)}$  \kf{}}\\
             \midrule 
	     \addlinespace
	     Tacnode &
            \makecell{$(\alpha:1)\in\PP^1$ if for some $k$\\  $(-\alpha)^{(n+2)k}=(-1)^k$, \\
                $(n+2)\alpha^2 a+\alpha b-(n+2)=0$.}\\
            \addlinespace
	    \midrule
	    \addlinespace
		(3,4,5)-cusp &
            \makecell{$(1:0)\in\PP^1$ for any $a$,\\
                $(0:1)\in\PP^1$ if $a=0$.} \\
            \addlinespace
	    \midrule
	    \addlinespace
            (2,5)-cusp &
            \makecell{$(1:0)\in\PP^1$ for any $a,b$,\\
                $(0:1)\in\PP^1$ if $a=0$.} \\
            \addlinespace
	    \midrule
	    \addlinespace
            \makecell{Cusp with\\ smooth branch} &
            \makecell{$(\alpha:1)\in\PP^1$ if for some $k$ \\
                $\alpha^{(n+2)k}=1$,\\
                $(n+2)\alpha=a$.} \\
            \addlinespace
	    \midrule
	    \addlinespace
            \makecell{Ordinary\\ triple point} &
            \makecell{$(\alpha:-1)\in\PP^1$ if for some $k$\\
                $(\alpha+1)^{(n+2)k}=(1+b+a)^k$,\\
                $\alpha^{(n+2)k}=a^k$.}\\
            \addlinespace
	    \midrule
	    \addlinespace
            Two cusps &
            \makecell{$(0:1)\in\PP^1$ if $a=0$, \\ $(1:0)\in\PP^1$ if $b=0$, \\ $(\frac{a}{n+2}:1)$ if $ab=(n+2)^2$.}\\
            \addlinespace
	    \midrule
	    \addlinespace
            \makecell{One cusp,\\ one node} &
            \makecell{$(1:0)\in\PP^1$ if for some $k$, $(b+1)^k=1$,\\ $(\frac{a}{n+2}:1)\in\PP^1$ if for some $k$,\\ $(a-(n+2))^{(n+2)k}=a^{(n+2)k}(b+1)^k$.}\\
            \addlinespace
	    \midrule
	    \addlinespace
            Two nodes &
            \makecell{ $(-1:\beta)\in\PP^1$ if for some $k$\\
                $ (-a)^k=\beta^{(n+2)k}$,\\
                $b^k(\beta+1)^{(n+2)k}=(-1)^k(\beta+c)^{(n+2)k}$.}\\
            \addlinespace
	    \bottomrule
        \end{tabular}
	\vspace{.5cm}
        \caption{Khovanskii-finiteness of non-degenerate rational curves of degree $n+2$ in $\PP^n$ having arithmetic genus 2 ($n\geq3$)}
        \label{tab:loci}
    \end{table}
    \begin{proof}
	    We proved part \ref{part:uno} for curves with a tacnode and curves with two nodes in Examples~\ref{ex:tacnode-param} and \ref{ex:two-nodes-param}.
        The parameterizations for the other cases are obtained in the same manner.

	For part \ref{part:dos}, we outline our general argument here.
        The proof is done by following these steps for each singularity family.
        \begin{step}
            \item Given a singularity configuration, we look up the parameterization for the set of representative curves in Table~\ref{tab:param}.
		    This parameterization involves a number of parameters (e.g.\ two parameters $a,b$ for curves with a tacnode).\label{step:1}
            \item The coordinate ring of a curve in this family is of the form $R(L)=\bigoplus_{k\geq0}L^k$ where $L\subset\KK[x,y]_{n+2}$ can be read off the corresponding parameterization.
            In this step we compute a basis of $L^k$ for $k\geq1$.
            We have recorded these bases in Table~\ref{tab:param2}.
            \item To determine which valuations $\val_Q$ from Construction~\ref{const:val-family} are \kf{} for a given curve, we use Theorem~\ref{thm:kf-criterion}. Using the explicit bases for $L^k$ from the previous step, we determine all points $(\alpha:\beta)\in\PP^1$ such that $\val_{(\alpha:\beta)}$ is \kf{} with respect to $R(L)$.
		    This yields a system of equations in terms of $\alpha,\beta$ and the parameters for each $k\geq1$.
	    A solution to this system for a specific $k$ gives a \kf{} valuation for a single curve.
        \end{step}
    \end{proof}

        \begin{table}
		\begin{adjustwidth}{-2cm}{-2cm}
       \centering
			\begin{tabular}{ l l }
                \toprule
                \bf \makecell{Singularity\\Type} &  \bf Basis for $L^k$ \\
\midrule\addlinespace
                Tacnode & 
                \makecell{
                    $\{x^{(n+2)k}-(-1)^k kbx+(-1)^k, x^{(n+2)k-1}+(-1)^k ax,x^{(n+2)k-2},\ldots, x^2\}$
	    }\\

\addlinespace\midrule\addlinespace

                (3,4,5)-cusp & 
                \makecell{
                    $\{x^{(n+2)k}+kax^{(n+2)k-2},x^{(n+2)k-3},\ldots,x,1\}$
                }\\

\addlinespace\midrule\addlinespace

                (2,5)-cusp &
                \makecell{
                    $\{x^{(n+2)k}-kax^{(n+2)k-3},x^{(n+2)k-2}-bx^{(n+2)k-3},x^{(n+2)k-4},\ldots,x,1\}$
                }\\

\addlinespace\midrule\addlinespace

                \makecell{Cusp with\\ smooth branch} & 
                \makecell{
                    $\{x^{(n+2)k}-kax^{(n+2)k-1}+(-1)^k,x^{(n+2)k-2},x^{(n+2)k-3},\ldots,x\}$
                }\\

\addlinespace\midrule\addlinespace

                \makecell{Ordinary\\ triple point} & 
                \makecell{
                    $\{x^{(n+2)k}+\left((1+b+a)^k-(1+a^k)\right)x+a^k,x^{(n+2)k-1}-x,\ldots,x^3-x,x^2-x\}$
                }\\

\addlinespace\midrule\addlinespace

                Two cusps &
                \makecell{
                    $\{x^{(n+2)k}-kax^{(n+2)k-1},x^{(n+2)k-2},\ldots,x^2,-kbx+1\}$
                }\\

\addlinespace\midrule\addlinespace

                \makecell{One cusp,\\ one node} & 
                \makecell{
                    $\{x^{(n+2)k}-kax^{(n+2)k-1}+(ka-1)x, x^{(n+2)k-2} -x,\ldots,x^2-x,((b+1)^k-1)x+1\}$
                }\\

\addlinespace\midrule\addlinespace
                Two nodes &
                \makecell{
                    $\left\{ (-a)^k x^{dk}-\frac{b^k-a^k-(-1)^k (c^{dk}-a^k b^k)}{b^k-(-1)^k c}x^{dk-1}+1,-\frac{b^k-(-1)^k c^i}{b^k-(-1)^k c}x^{dk-1}+x^{dk-i} \big| 2\leq i\leq dk-1 \right\},$\\
                    or\\
                    $\left\{ (-a)^k x^{dk}-\frac{b^k-a^k-(-1)^k (c^{dk}-a^k b^k)}{b^k-(-1)^k c^2}x^{dk-2}+1,-\frac{b^k-(-1)^k c^i}{b^k-(-1)^k c^2}x^{dk-2}+x^{dk-i} \big| 1\leq i\leq dk-1 , i\neq2\right\},$\\
                where $d=n+2$ (either works if the denominators do not vanish)
                }\\
\addlinespace\bottomrule
            \end{tabular}
    \end{adjustwidth}
	    \vspace{.5cm} 
            \caption{
		    Bases for $L^k$ for rational curves in Table \ref{tab:param}. To save space, we have set $y=1$.}
	    \label{tab:param2}
        \end{table}

        We illustrate the steps of the above proof in the next example.
	\begin{example}[Cusp with smooth branch]\label{ex:cwsb}
            Consider the family of rational curves of degree $n+2$ in $\PP^n$ having a cusp with smooth branch.
            According to Table~\ref{tab:param}, the set of representative curves is a one dimensional family. Each curve $X$ has a coordinate ring $R(L)=\bigoplus_{k\geq 0}L^k$ where
            \[
                L=\mathrm{span}\{x^{n+2}-ax^{n+1} y-y^{n+2},x^n y^2,x^{n-1} y^3,\ldots,xy^{n+1}\}\subset\KK[x,y]_{n+2},
            \]
            for some $a\in\A^1$.
            We claim that a basis for $L^k$ is given by
	    \begin{align*}
                \mathcal{B}_k=\{x^{(n+2)k}-kax^{(n+2)k-1}y+(-1)^k y^{n+2},\\
	\qquad	x^{(n+2)k-2}y^2,x^{(n+2)k-3}y^3,\ldots,xy^{(n+2)k-1}\}.
	\end{align*}
            To prove this we induct on $k$.
            The base case obviously holds. Assume the claim holds for some $k\geq1$.
            The subspace $L^{k+1}$ is the span of all elements of the form $f\cdot f'$ where $f\in\mathcal{B}_1$ and $f'\in\mathcal{B}_k$.
            It is easy to check that all such products are in $\mathcal{B}_{k+1}$, that is $L^{k+1}\subset\mathrm{span}\left( \mathcal{B}_{k+1} \right)$.
            It is thus enough to show that $L^{k+1}$ and $\mathrm{span}(\mathcal{B}_{k+1})$ have the same dimension.
            Since elements of $\mathcal{B}_{k}$ are $\KK$-linearly independent, we have $\dim_{\KK}\left( \mathrm{span}(\mathcal{B}_{k}) \right)=(n+2)k-1$ for all $k\geq1$.
            On the other hand, we know that a non-degenerate rational curve of degree $d$ in $\PP^n$ having arithmetic genus $d-n$ is 3-regular~\cite[Theorem 1]{regularity}.
            In our case, this means that for $k\geq2$ we have
            \[
                \dim_{\KK}L^k= h^0\left(\CO_X(k)\right)=h_X(k),
            \]
            where $h_X(k)=(n+2)k-1$ is the Hilbert polynomial of $X$.
            The claim for the basis of $L^k$ is now proved.

            Next we determine for which points $(\alpha:\beta)\in\PP^1$, the condition $(\beta x-\alpha y)^{(n+2)k}\in L^k$ holds.
            Clearly we cannot have $\beta=0$, so we assume $\beta=1$.
	    Expanding this expression and using the basis for $L^k$ immediately gives
            \begin{align*}
                \begin{split}
                    \alpha^{(n+2)k}&=1,\\
                    (n+2)\alpha&=a.
                \end{split}
            \end{align*}
        \end{example}

	\begin{rem}\label{rem:locus}The information of Table \ref{tab:loci} allows us to describe the locus of curves in each stratum of representative curves that admit a \kf{} valuation. Indeed, as discussed following Construction \ref{const:val-family}, in testing Khovanskii-finiteness we only need to consider valuations of the form $\val_Q$. By consulting Table \ref{tab:loci}, we see that $(3,4,5)$-cusps and $(2,5)$-cusps are always \kf{}. In the case of two cusps, the \kf{} locus is a union of three curves in $\A^2$. In all the remaining cases, one obtains that the \kf{} locus is an infinite but countable union of proper subvarieties.
For a cusp and node, the equations are exactly the conditions on $a,b$ as written in the table as $k$ varies. 	For the cases of a tacnode or cusp with smooth branch, one obtains the equations by letting $\alpha$ range over all possible roots of unity.

The biggest challenge is in the case of an ordinary triple point or two nodes. Here, one may obtain equations for each fixed value of $k$ by eliminating $\alpha$ or $\beta$ from the system of equations. This is unfortunately not very explicit. Alternatively, by taking $k$-th roots, one may replace the system of equations in the case of an {\bf ordinary triple point} by the following systems:
\begin{align*}
	(\alpha+1)^{n+2}&=u(1+b+a)\\
	\alpha^{n+2}&=va
\end{align*}
where $u,v$ are roots of unity. Treating $u,v$ as indeterminates and eliminating $\alpha$, one obtains a single polynomial equation $H_{u,v}(a,b)=0$ in $a,b,u,v$. The \kf{} locus is contained in the union of the vanishing loci of the $H_{u,v}$ as we substitute in roots of unity for $u$ and $v$. In the special case $n=3$, one computes (e.g.~with Macaulay2 \cite{M2}) that
\[
	H_{u,v}=\left ( (u-v)a+ub+u-1\right)^5-625uva(a+b+1)I_{u,v}
\]
where
\[
	I_{u,v}=\left(ub+(u+\frac{3}{2}v)a+(u+\frac{3}{2})\right)^2-\frac{5}{4}(1+6va+v^2a^2).
\]
One may show that $H_{u,v}$ remains irreducible after substituting in any roots of unity for $u,v$, which shows that in the case $n=3$, the \kf{} locus is exactly the union of the curves $H_{u,v}=0$ as $u,v$ range over all roots of unity.

A similar analysis may be performed in the case of {\bf two nodes}.
Again by taking $k$-th roots, we obtain the systems
\begin{align*}
	au&=\beta^{n+2}\\
	bv(\beta+1)^{n+2}&=(\beta+c)^{n+2}
\end{align*}
where $u,v$ are roots of unity. Treating $u,v$ as indeterminates and eliminating $\beta$, one obtains a single polynomial equation $H'_{u,v}(a,b,c)=0$ in $a,b,c,u,v$. The \kf{} locus is contained in the union of the vanishing loci of the polynomials $H'_{u,v}(a,b,c)$ as we substitute in roots of unity for $u$ and $v$. Already in the case $n=3$, the polynomial $H'_{u,v}(a,b,c)$ is quite complicated, containing over $100$ terms.
	\end{rem}

	\begin{rem}
		Given a non-degenerate rational curve $X$ of degree $d=n+2$ in $\PP^n$ of arithmetic genus 2, Theorem \ref{mainthm:1} gives us a criterion for testing if a given valuation $\val_Q$ is \kf{}. In the case that $\val_Q$ is in fact \kf{}, we obtain a toric degeneration of $X$ by \cite[Theorem 1.1]{anderson}. The central fiber of this degeneration must be a projective toric curve of arithmetic genus two; the only possibilities are a rational curve with two simple cusps, a rational curve with a $(2,5)$-cusp, or a rational curve with a $(3,4,5)$-cusp.

		In general, the toric degeneration will not take place in $\PP^n$, but in some weighted projective space. The homogeneous coordinate ring for the central fiber of the degeneration is simply the semigroup ring of the value semigroup of $\val_Q$ on $R(L)$. In principal, one may determine this semigroup using the bases we list in Table \ref{tab:param2}, but in practice, this appears to be rather subtle.
	\end{rem}
    \subsection{Deciding Khovanskii-finiteness}\label{sec:bound}
    Algorithm~\ref{alg:naive} provides a non-effective procedure for deciding whether or not a valuation on the coordinate ring of a rational curve is \kf{}.
    The issue is that if the valuation is not \kf{} then the process never terminates and so we called it the \emph{naive algorithm}.
    The main result of this section is that for non-degenerate rational curves of arithmetic genus 2, there is an upper bound for the number of iterations needed in the naive algorithm.

    \begin{theorem}\label{thm:general-bound}
        Let $X\subset\PP^n$ be a non-degenerate rational curve of degree $d$ having arithmetic genus 2 with coordinate ring $R(L)=\bigoplus_{k\geq0}L^k$ where $L\subset\KK[x,y]_{d}$ is a subspace and let $(\alpha:\beta)\in\PP^1$.
        Assume further that $L$ is defined over a number field $F$ with $[F:\QQ]=\ell$.
        Then $\val_{(\alpha:\beta)}$ is \kf{} with respect to $R(L)$ if and only if there exists
	\[1\leq k\leq \max \left\{2(d-n-1),\left( \left(96d^3 \ell\right)^2+2 \right)^2\right\}\]
    such that $(\beta x-\alpha y)^{dk}\in L^k$.
    \end{theorem}
    We will prove this theorem later in this section.
  
    \begin{rem}
	    As in \cite[Remark 3.6]{ilten-wrobel}, checking if $R(L)$ admits \emph{any} homogeneous Khovanskii-finite valuation amounts to testing if the image of the $kd$-power map $\PP^1=\PP(\KK[x,y]_1)\to\PP^{kd}=\PP(\KK[x,y]_{kd})$ intersects $L^k$ for some $k$.  By the above theorem, it suffices to check this for the range of $k$ specified in the theorem.
    \end{rem}
    \begin{rem}
	    We saw in Remark \ref{rem:locus} that in general, the locus of degree $d$ rational curves of arithmetic genus $g$ in $\PP^n$ admitting a Khovanskii-finite valuation is not a Zariski-closed set, but rather a countable union of Zariski-closed sets. However, by Theorem \ref{thm:general-bound}, we see that for any number field $F$, the set of $F$-rational points in this locus \emph{does} form a Zariski-closed set.
    \end{rem}

    Before proving the theorem in general, we determine similar bounds for the degree $n+2$ representative curves in $\PP^n$ from Table~\ref{tab:param}.
    The following preliminary result regarding roots of unity will be used repeatedly.

    \begin{lemma}
        \label{lem:root-unity}
        Let $\zeta$ be an algebraic number of degree $\ell$ over $\QQ$.
        Then $\zeta$ is a root of unity if and only if $1\in\{\zeta,\zeta^2,\ldots,\zeta^{\ell^2+2}\}$.
    \end{lemma}
    \begin{proof}
        One direction is clear.
        For the other, suppose $\zeta$ is a primitive $m$-th root of unity.
        Then the minimal polynomial of $\zeta$ is
        $\Phi_m(x)\in\QQ[x]$
        where $\Phi_m$ is the $m$-th cyclotomic polynomial and
        $\ell=\deg\Phi_m=\varphi(m)$ ($\varphi$ is the Euler's totient function).
        If $m$ is either $2$ or $6$, then $\ell=1$ or 2 respectively and the statement holds.
        Assume $m\neq2,6$.
        We use the known bound $\varphi(m)\geq\sqrt{m}$ for $m\neq2,6$ (see~\cite[p. 9]{totient-handbook}).
        If $m>\ell^2$, then we have
        \begin{equation*}
            \varphi(m)\geq\sqrt{m}>\ell,
        \end{equation*}
        a contradiction.
        Hence $\zeta$ is a primitive $m$-th root of unity for some $m\leq \ell^2$.
    \end{proof}

    The next proposition derives upper bounds for the number of iterations in the naive algorithm for each family of representative curves.
    \begin{proposition}\label{prop:rep-curve-bounds}
        Let $X$ be one of the representative curves in Table~\ref{tab:param} with coordinate ring $R(L)$ where $L\subset\KK[x,y]_{n+2}$.
        Let $(\alpha:\beta)\in\PP^1$ and assume the parameters from  the table are defined over a number field $F$ with $[F:\QQ]=\ell$.
        Then for each singularity family, Table~\ref{tab:bounds} lists an upper bound for the number of iterations needed in the naive algorithm.
        In particular, $\left(16(n+2)^4 \ell^2+2\right)^2$ serves as an upper bound for all representative curves.
    \end{proposition}

    \begin{table}
        \begin{adjustwidth}{-2cm}{-2cm}
            \centering
            \begin{tabular}{ ll }
                \toprule
                \bf \makecell{Singularity\\ Configuration}
                & \bf \makecell{Bound for the\\ Naive Algorithm} \\
                \midrule
                Tacnode & $(2\ell)^2+2$  \\
                (3,4,5)-cusp & $1$ \\
                (2,5)-cusp & $1$ \\
                Cusp with smooth branch & $\ell^2+2$ \\
                Ordinary triple point & $\left(16(n+2)^4\ell^2+2\right)^2$ \\
                Two cusps & $1$ \\
                One cusp, one node & $\ell^2+2$\\
                Two nodes & $\left(16(n+2)^4\ell^2+2\right)^2$ \\
                \bottomrule
            \end{tabular}
        \end{adjustwidth}
        \vspace{1cm}
        \caption[Bounds for the naive algorithm]{Bounds for $k$ in Algorithm~\ref{alg:naive} for a representative curve parameterized over a number field of degree $\ell$ over $\QQ$}
        \label{tab:bounds}
    \end{table}

    \begin{proof}
        The curve $X$ belongs to one of the families in Table~\ref{tab:param}, i.e.\ we are given the corresponding parameters in the second column of this table. 
        The corresponding equations in Table~\ref{tab:loci} then determine Khovanskii-finiteness as follows:
        For $\val_{(\alpha:\beta)}$ to be \kf{} with respect to $R(L)$, the parameters of $X$ must satisfy these equations for at least one $k\in\NN$.
        This is the same $k$ as in Theorem~\ref{thm:kf-criterion} and consequently the same $k$ as the iteration number in Algorithm~\ref{alg:naive}.
        All we need to do then is to specify a bound $M$ such that the equations are satisfied for some value $k\in\NN$ if and only if they are satisfied for some $k\leq M$.
        For families whose equations in Table~\ref{tab:loci} do not involve $k$, i.e. (3,4,5)-cusp, (2,5)-cusp and two cusps, this bound is 1.
        We prove the bounds for the other families.
        The last statement of the proposition then follows by taking the maximum of the bounds in Table~\ref{tab:bounds}.

        \textbf{Tacnode:} Let $X$ be a representative curve in this family.
        By the corresponding equations in Table~\ref{tab:loci}, if $\beta=0$, $\val_{(\alpha:\beta)}$ is not \kf{}.
        We may thus assume $\beta=1$.
        Then $\alpha$ and the given parameter $a$ must satisfy
        \begin{align*}
            \left( -(-\alpha)^{n+2} \right)^k&=1,\\
            (n+2)\alpha^2 a+\alpha b-(n+2)&=0.
        \end{align*}
        The second equation shows that $\alpha$ lives in at most a quadratic extension of $F$.
        After extending $F$ we may assume $\alpha\in F$ and $[F:\QQ]=2\ell$.
        By Lemma~\ref{lem:root-unity}, this system is satisfied for some $k\in\NN$ if and only if it is satisfied for some $k\leq (2\ell)^2+2$ as claimed.

	The bounds for curves with a \textbf{cusp with smooth branch} and curves with \textbf{one cusp and one node} can be proved similarly.
        In each case we need to check whether an algebraic number defined over $F$ is a $k$-th root of unity.

        \textbf{Ordinary triple point:} Let $X$ be a representative curve having an ordinary triple point with given parameters $a,b$ as in Table~\ref{tab:param}.
        According to Table~\ref{tab:loci}, for $\val_{(\alpha:\beta)}$ to be \kf{}, we must have $\beta\neq 0$.
        We may assume $\beta=-1$.
        Then $\alpha$ and the parameters $a,b$ must satisfy
        \begin{align}\label{eq:otp-eq}
        \begin{split}
            (\alpha+1)^{(n+2)k}&=(1+b+a)^k,\\
            \alpha^{(n+2)k}&=a^k,
        \end{split}
        \end{align}
        which is equivalent to demanding that there exist integers $k_1,k_2$ such that
        \[
            \left( \frac{(\alpha+1)^{n+2}}{1+b+a} \right)^{k_1}=\left( \frac{\alpha^{n+2}}{a} \right)^{k_2}=1.
        \]
        The denominators are nonzero because of the conditions on $a,b$ in Table~\ref{tab:param}.
        Then $k=k_1 k_2$ can be used to satisfy equation~\eqref{eq:otp-eq}.
        We will show that an algebraic number $\alpha$ satisfying~\eqref{eq:otp-eq} for some $k\in\NN$ lives in an extension field of degree at most $4(n+2)^2$ over $F$.
        Then after extending $F$, we may assume $\alpha\in F$, $[F:\QQ]=4(n+2)^2\ell$.
        By Lemma~\ref{lem:root-unity}, it is enough to consider the integers $k_1,k_2\leq \left(4(n+2)^2 \ell\right)^2+2$ and thus $\left(16(n+2)^4\ell^2+2\right)^2$ serves as an upper bound for the values of $k$ that we need to check.

        It remains to prove the bound on the degree of $\alpha$ over $F$.
        Taking $(n+2)k$-th roots in~\eqref{eq:otp-eq} gives
        \begin{align*}
            \alpha+1&=(1+b+a)^{\frac{1}{n+2}}z',\\
            \alpha&=a^{\frac{1}{n+2}}z,
        \end{align*}
        where $z,z'$ are some roots of unity.
        By eliminating $\alpha$, we get the equation
        \begin{equation}\label{eq:otp-linear}
        a^{\frac{1}{n+2}}z-(1+b+a)^{\frac{1}{n+2}}z'+1=0,
        \end{equation}
        in $z,z'$ with coefficients in some extension $F'$ with $[F':F]\leq (n+2)^2$.
        By a result of W.~Ruppert~\cite[Corollary 6]{ruppert}, any solution $(z,z')$ to this equation in roots of unity has degree at most 4 over $F'$.
        Consequently, $\alpha$ is defined over an extension field of degree at most $4(n+2)^2$ over $F$.

        The proof of the bound for curves with \textbf{two nodes} is similar to the case of an ordinary triple point.
	Here, by Table~\ref{tab:loci} we are considering the system of equations 
	\begin{align}\label{eqn:tn}
        \begin{split}
		(-a)^k&=\beta^{(n+2)k},\\
		b^k(\beta+1)^{(n+2)k}&=(-1)^k(\beta+c)^{(n+2)k}.
        \end{split}
        \end{align}
After taking $(n+2)k$-th roots and eliminating $\beta$, we arrive at the equation
\[
    (-b)^{\frac{1}{n+2}}\left((-a)^{\frac{1}{n+2}}z+1\right)=\left((-a)^{\frac{1}{n+2}}z+c\right)z'
\]
in the roots of unity $z,z'$. The coefficients in this equation live in some extension $F'$ with $[F':F]\leq (n+2)^2$. Again by \cite[Corollary 6]{ruppert}, any solution $(z,z')$ to this equation in roots of unity has degree at most 4 over $F'$, and $\beta$ is defined over an extension field of degree at most $4(n+2)^2$ over $F$.

In a similar fashion to the case of an ordinary triple point, we may then use the bound on the degree of $\beta$ and Lemma \ref{lem:root-unity} to conclude that it suffices to check \eqref{eqn:tn} for $k\leq \left(16(n+2)^4\ell^2+2\right)^2$.
    \end{proof}

    We now move on to consider arbitrary non-degenerate degree $d$ rational curves $X\subset \PP^n$ of arithmetic genus two. Consider a rational curve $X=\Proj(R(L))$ for some $L\subset \KK[x,y]_d$ with $\PP(L^\perp)\cap C_d=\emptyset$. Adapting terminology used in Proposition \ref{prop:project}, we say that $L$ has \emph{ramification defined over a number field $F$} if the corresponding projection $\pi:C_d\to X$ has ramification defined over $F$. Note that $X$ and $\pi$ being defined over $F$ is equivalent to $L$ being defined over $F$.

    \begin{lemma}\label{lemma:galois}
Let $X=\Proj(R(L))$ be as above for some $L$ defined over a number field $F$, and assume that $X$ has arithmetic genus two.
Then there is an extension $F'$ of $F$ of degree at most $24$ such that $L$ has ramification defined over $F'$.
    \end{lemma}
    \begin{proof}
Let $F'$ be the Galois closure of the field of definition of the points of $\pi^{-1}(\sing(X))$; clearly $L$ has ramification defined over $F'$. Since $L$ (and thus $\pi$) is defined over $F$, there is an injection of the Galois group of $F'$ over $F$ into the symmetric group on the elements of $\pi^{-1}(\sing(X))$. On the other hand, by Theorem \ref{thm:singularities-lambda}, $\pi^{-1}(\sing(X))$ has at most $4$ elements. Hence, the degree of $F'$ over $F$ is at most $4!=24$.
    \end{proof}

    We now prove the main result of the section.

    \begin{proof}[{Proof of Theorem~\ref{thm:general-bound}}]
	    By Theorem \ref{thm:kf-criterion}, the condition of the theorem is clearly sufficient for Khovanskii-finiteness. We will show that it is also necessary. Our strategy is to reduce to the setting of Proposition \ref{prop:rep-curve-bounds}.
	    By Lemma \ref{lemma:galois}, we may assume that $L$ has ramification defined over $F$ after multiplying the degree $\ell$ of $F$ by $24$.
	    By Proposition \ref{prop:project}, the curve $X=\Proj(R(L))$ is obtained via an isomorphic projection of a degree $d$ curve $X'=\Proj(R(L'))$ in $\PP^{d-2}$. Furthermore, since $L$ has ramification defined over $F$, so does $L'$.

In \S\ref{sec:param} we saw that $X'$ is projectively equivalent over $\KK$ to a representative curve. We claim that in fact, this projective equivalence can be realized over an extension of $F$ of degree at most $d$. Indeed, 
To show in \S\ref{sec:param} that $X'$ is projectively equivalent over $\KK$ to a representative curve, we used elements of $\pgl{2}$ whose action on $\PP^{d}$:
\begin{enumerate}
	\item Moved points of $\pi^{-1}(\sing(X'))$ to special points of $C_d\subset \PP^d$;\label{move:one}
	\item Moved at most one point on a tangent or secant line of $C_d$ to a special position.\label{move:two}
\end{enumerate}
Since $L'$ has ramification defined over $F$, the points we are trying to move in items \ref{move:one} and \ref{move:two} are defined over $F$; the points we are trying to move them to are as well. To move points as in item \ref{move:one}, it suffices to consider $\pgl{2}$ elements defined over $F$, since $\mathrm{PGL}(2,F)$ acts transitively on the $F$-rational points of $C_d$. On the other hand, it is a straightforward computation to check that the desired transformation of item \ref{move:two} can be obtained after extending $F$ by degree $d$.

Thus, after possibly extending $F$ again (and multiplying the degree of $F$ by $d$), we will now be assuming that $X'$ is a representative curve.
It now follows from Proposition \ref{prop:rep-curve-bounds} 
that $\val_{(\alpha:\beta)}$ is \kf{} with respect to $R(L')$ if and only if there exists $1\leq k \leq ((96d^3\ell)^2+2)^2$ such that $(\beta x-\alpha y)^{dk}\in (L')^k$. (Recall that we had to multiply the degree of $F$ by $24\cdot d$).

	    Since $X$ is a projection of $X'$, we have $R(L)\subset R(L')$.
By \cite[Theorem 1]{regularity}, $X$ is $d-n+1$ regular, and $X'$ is $3$-regular.
 Since these rings have the same Hilbert polynomials, by the above regularity, they must agree in degrees $k \geq (d-n)$. In other words, $L^k=(L')^k$ for $k\geq (d-n)$. 

To finish the proof, suppose that $\val_{(\alpha:\beta)}$ is \kf{} with respect to $R(L)$. Since $L^k=(L')^k$ for $k\geq (d-n)$, $\val_{(\alpha:\beta)}$ is \kf{} with respect to $R(L')$. By the above, there exists $1\leq k \leq ((96d^3\ell)^2+2)^2$ such that $(\beta x-\alpha y)^{dk}\in (L')^k$. If $k\geq d-n$, then we also have $(\beta x-\alpha y)^{dk}\in L^k$. If instead $k<d-n$, then there exists a natural number $m$ such that $(d-n)\leq mk \leq 2(d-n-1)$, and we have $(\beta x - \alpha y)^{dmk}\in (L')^{mk}=L^{mk}$. It follows that the condition of the theorem is necessary for Khovanskii-finiteness.
    \end{proof}

    \section{Rational Curves with Unique Unibranch Singularity}\label{sec:unibranched}

    We now show that rational curves with a single unibranch singularity are often Khovanskii-finite. For any rational curve $X$, let $Q$ be a unibranch singular point. The point $Q$ determines a valuation $\val_{\phi^{-1}(Q)}$, where $\phi:\PP^1\to X$ is the normalization map. Indeed, since $Q$ is unibranch, the preimage $\phi^{-1}(Q)$ is a single point, so $\val_{\phi^{-1}(Q)}$ is well-defined.

    \begin{theorem}
        \label{mainthm:3}
	Let $X\subset \PP^n$ be a non-degenerate rational curve of degree $d$ and arithmetic genus $g\leq d/2$.
        Assume that $X$ has only one singular point $Q$, and this is a unibranch singularity.
        Then the valuation corresponding to $Q$ is Khovanskii-finite with respect to $X$.
    \end{theorem}
    \begin{proof}
	    By Proposition \ref{prop:project} and Proposition \ref{prop:same}, we may assume without loss of generality that $g\geq d-n$. Let $R(L)$ be the projective coordinate ring of $X\subset \PP^n$, where $L\subset \KK[x,y]_d$. We consider the valuation semigroup $\Sigma\subset \ZZ_{\geq 0}$ of $\CO_{X,Q}$ with respect to the valuation $\val$ corresponding to $Q$. Since the only singularity of $X$ is $Q$, the delta-invariant of $Q$ is equal to the arithmetic genus $g$. Furthermore, the delta-invariant is exactly the size  of the complement $\Sigma^c$ of $\Sigma$ in $\ZZ_{\geq 0}$.

	    Since the valuation of the conductor of a unibranch singularity never exceeds twice the delta-invariant, every one of the $g$ elements $m\in \Sigma^c$ must satisfy $m<2g$.
	    On the other hand, the image $\Lambda$ of $L$ under $\val$ consists of exactly $n+1$ elements of the set $\{0,1,\ldots,d\}$. Its complement $\Lambda^c$ thus consists of $d-n\leq g$ elements. Since $\Lambda\subset \Sigma$ and $2g\leq d$ we conclude that $\Sigma^c=\Lambda^c$. It follows that $\Sigma$, and thus $\Lambda$ has an element of valuation $d$. But this implies that $\val$ is Khovanskii-finite with respect to $X$ by Theorem \ref{thm:kf-criterion}.
    \end{proof}
    \begin{rem}
	    If we assume that $d<2n$, then the hypothesis in Theorem \ref{mainthm:3} on the arithmetic genus is automatically satisfied. Indeed, $g\leq d-n$ by Castelnuovo's bound, see e.g.~\cite[Theorem 1.1]{buczynski-singular}. But then $2g\leq 2d-2n=d+(d-2n)<d$.
    \end{rem}

    \begin{rem}
	    The bound on $g$ in Theorem \ref{mainthm:3} is tight. Indeed, consider 
	    \[L=\langle x^7+y^7,x^3y^4,x^2y^5,xy^6\rangle \subset \KK[x,y]_7\]
	    and $X=\Proj(R(L))\subset \PP^3$. The curve $X$ has degree $7$, and has a single unibranch singularity at the image of $(1:0)$. The arithmetic genus of this curve is $4$, so we have $2g=d+1$. However, it is straightforward to check that $R(L)$ is \emph{not} Khovanskii-finite with respect to the valuation $\val_{(1:0)}$.
    \end{rem}

    \printbibliography
    \typeout{get arXiv to do 4 passes: Label(s) may have changed. Rerun}
\end{document}